\documentclass[11pt]{article}
\usepackage{amssymb,amsmath,amsfonts,amsthm,epsfig,latexsym,color}
\usepackage{amsfonts,amsmath,amsthm, amssymb, amscd, mathrsfs}
\usepackage{latexsym, euscript, epic, eepic}
\usepackage{graphicx}
\usepackage{verbatim}
\usepackage{fourier}

\usepackage[colorlinks=true,backref=page]{hyperref}

\makeatletter
\newcommand{\subsectionruninhead}{\@startsection{subsection}{2}{0mm}
{-\baselineskip}{-0mm}{\bf\large}}
\newcommand{\subsubsectionruninhead}{\@startsection{subsubsection}{3}{0mm}
{-\baselineskip}{-0mm}{\bf\normalsize}}
\makeatother

\newtheorem*{theorem*}{Theorem}

\newtheorem{theoremalph}{Theorem}

\newtheorem*{proposition*}{Proposition}
\newtheorem*{corollary*}{Corollary}
\newtheorem*{claim*}{Claim}
\newtheorem*{remark*}{Remark}
\newtheorem*{problem*}{Problem}
\newtheorem{theorem}{Theorem}[section]
\newtheorem{conjecture}{Conjecture}[section]
\newtheorem{proposition}[theorem]{Proposition}

\newtheorem{lemma}[theorem]{Lemma}

\newtheorem{claim}[theorem]{Claim}
\theoremstyle{definition}
\newtheorem{definition}[theorem]{Definition}
\newtheorem{remark}[theorem]{Remark}

\numberwithin{equation}{section}

\newcommand{\supp}{\operatorname{\,Supp}}

\newcommand{\orb}{\operatorname{Orb}}

\newcommand{\sing}{\operatorname{Sing}}

\begin{document}

\title{Measures of intermediate pressures for geometric Lorenz attractors}

\author{Yi Shi\footnote{Y. Shi was partially supported by National Key R\&D Program of China (2021YFA1001900) and NSFC (12071007, 11831001, 12090015).}\, \,  
 and  
Xiaodong Wang\footnote{X. Wang was partially supported by National Key R\&D Program of China (2021YFA1001900), NSFC (12071285) and Innovation Program of Shanghai Municipal Education Commission (No. 2021-01-07-00-02-E00087).}}

\maketitle

\begin{abstract}
	Pressure measures the complexity of a dynamical system concerning a continuous observation function.
    A dynamical system is called to admit the intermediate pressure property if for any observation function, the measure theoretical pressures of all ergodic measures form an interval.
    We prove that the intermediate pressure property holds for $C^r (r\geq 2)$ generic geometric Lorenz attractors while it fails for $C^r (r\geq 2)$ dense geometric Lorenz attractors, which gives a sharp contrast. Similar results hold for $C^1$ singular hyperbolic attractors. 
\end{abstract}


\section{Introduction}
Entropy, which reflects the complexity of dynamical systems, has been proved to be the most important invariant in ergodic theory and dynamical systems since it was first introduced by Kolmogorov~\cite{Kol-58,Kol-59} in the 1950s. In his milestone work~\cite{Katok}, Katok proved that if a $C^r (r>1)$ surface diffeomorphism has positive topological entropy,  then it admits  horseshoes  whose entropies accumulate to the topological entropy. 
A classical result  is that a  horseshoe $\Lambda$ of a diffeomorphism $f$ admits the {\it intermediate entropy property}: for any non-negative constant $h$ smaller than the topological entropy $h_{\rm top}(\Lambda)$ of the horseshoe, there exists an ergodic measure $\mu$ of $f$ whose measure theoretical entropy $h_\mu(f)$ equals $h$.
As a consequence of Katok's work~\cite{Katok}, every $C^r (r>1)$ surface diffeomorphism $f$ admits the intermediate entropy property.
Katok raised the following conjecture.

\begin{conjecture}[Katok]
	For $r\geq 1$, every $C^r$ diffeomorphism of a smooth closed Riemannian manifold admits the intermediate entropy property.
\end{conjecture}

Note that there are $C^0$ uniquely ergodic systems with positive entropy~\cite{HK,GW94,BCL}, thus such systems do not admit the intermediate entropy property.
Systems that Katok's conjecture holds true including: certain skew product diffeomorphisms~\cite{Sun-2010}, certain homogeneous systems~\cite{GSW}, affine transformations of nilmanifolds~\cite{HXX}, some partially hyperbolic diffeomorphisms with 1-dimensional center~\cite{Ures,YZ20} and star vector fields~\cite{LSWW}.
In particular, it was proved in~\cite{LSWW} that every horseshoe of a  vector field admits the intermediate entropy property, which gives the counterpart of the corresponding result for horseshoes of diffeomorphisms.
More progress on Katok's conjecture can be found in~\cite{Sun-10,Sun-12,Sun-2019}.

%

As a generalization of entropy, the notion of pressure was introduced by Ruelle~\cite{Ruelle} and studied in general case in Walters~\cite{Walters}. It measures the complexity of a dynamical system through adding the influence of an observation function. 
Let  $K$ be a compact metric space and $f\colon K\rightarrow K$ be a continuous transformation. 
Denote by $C(K,\mathbb{R})$ the space of all continuous functions from $K$ to $\mathbb{R}$ endowed with the  supremum norm. 
The topological pressure of $f$ (see definition in Section~\ref{Section:pressure}) is a map  
$$P_{\rm top}(f,\cdot)\colon C(K,\mathbb{R})\rightarrow\mathbb{R}\cup\{\infty\}$$ 
which has good properties relative to the structures on $C(K,\mathbb{R})$.
Denote by $\mathcal{M}_{\rm inv}(f)$ and $\mathcal{M}_{\rm erg}(f)$ respectively the space of invariant and ergodic probability measures of $f$. Given $\varphi\in C(K, \mathbb{R})$ and  $\mu\in \mathcal{M}_{\rm inv}(f)$, define the {\it measure theoretical pressure} of $\mu$ as
\[P_\mu(f,\varphi)=h_\mu(f)+\int \varphi {\rm d}\mu,\]
where $h_\mu(f)$ is the measure theoretical entropy of $\mu$.
By the variational principle~\cite[Theorem 9.10]{Wa} for pressure, one has that  
\[P_{\rm top}(f,\varphi)=\sup_{\mu\in\mathcal{M}_{\rm inv}(f)} P_\mu(f,\varphi)=\sup_{\mu\in\mathcal{M}_{\rm erg}(f)} P_\mu(f,\varphi).\]
Denote by
\[P_{\inf}(f,\varphi)=\inf_{\mu\in\mathcal{M}_{\rm inv}(f)} P_\mu(f,\varphi).\]
A dynamical system $(K,f)$ admits the {\it intermediate pressure property} if for any  $\varphi\in C(K, \mathbb{R})$ and for any $P\in \left(P_{\inf}(f,\varphi),P_{\rm top}(f,\varphi)\right)$, there exists $\mu\in \mathcal{M}_{\rm erg}(f)$ such that $P_{\mu}(f,\varphi)=P$.
\medskip

In this paper, we study the intermediate pressure problem for geometric Lorenz attractors. 
The Lorenz attractor, whose dynamics behaves in a chaotic way, was found by Lorenz~\cite{Lorenz} from an ordinary differential equation system in the study of meteorology.
Guckenheimer~\cite{guck} and Afra$\breve {\rm \i}$movi$\check{\rm c}$ {\it et al}~\cite{abs} introduced a geometric model for the strange attractor which is called the \emph{geometric Lorenz attractor} (see Definition~\ref{Def:Lorenz}) nowadays. The structure of geometric Lorenz attractor was well studied by Guckenheimer and Williams~\cite{gw,williams}.
Due to the existence of the singularity, the geometric Lorenz attractor is not hyperbolic and exhibits complicated dynamics. In the study of robustly transitive singular sets of 3-dimensional vector fields, Morales-Pacifico-Pujals~\cite{mpp} introduced the notion of singular hyperbolicity (see Definition~\ref{Def:SH}) which well describes the geometric structure of geometric Lorenz attractors. 
The notion of singular hyperbolicity was extended to higher dimension in~\cite{lgw,Me-Mo,zgw} and used to characterize the structure of chain recurrence classes which admit a homogeneous index condition of singularities for generic star vector fields in~\cite{sgw}.

In a recent joint work of X. Tian and the two authors~\cite{STW}, we characterize the space of ergodic measures of geometric Lorenz attractors. To be precise, we proved that for every integer $r\geq 2$ or $r=\infty$, the ergodic measure space is path connected with dense periodic measures for $C^r$ generic geometric Lorenz attractors while it is not connected for $C^r$ dense geometric Lorenz attractors. As an application, we study the multifractal analysis and large deviations of Gibbs measures for $C^r$ generic geometric Lorenz attractors in~\cite{STVW}, in which the main technique we apply is horseshoe approximation property rather than specification-like properties.

In this paper, we prove  that for every integer $r\geq 2$, intermediate pressure property holds for $C^r$ generic geometric Lorenz attractors  while it fails for $C^r$ dense geometric Lorenz attractors, which gives a sharp contrast.
Let $M^d$ (or $M$ for simplicity) be a $d$-dimensional closed smooth Riemannian manifold.
Denote by $\mathscr{X}^r(M)$ the space of $C^r$ vector fields on $M$ where $r\in\mathbb{N}\cup\{\infty\}$.

\begin{theoremalph}\label{Thm:Lorenz}
	For every $ r\in \mathbb{N}_{\geq 2}\cup\{\infty\}$, there exist a dense $G_\delta$ subset $\mathcal{R}^r$ and a dense subset $\mathcal{D}^r$ in $\mathscr{X}^r(M^3)$ such that
	\begin{itemize}
		 \item {\bf generic systems:} if $\Lambda$ is a geometric Lorenz attractor of $X\in \mathcal{R}^r$, then $\Lambda$ admits the  intermediate pressure property;
		 
		 \item {\bf dense systems:} if $\Lambda$ is a geometric Lorenz attractor of $X\in \mathcal{D}^r$, then $\Lambda$ does not admit the  intermediate pressure property, in other words, there exist  
		$\varphi\in C(\Lambda,\mathbb{R}) \text{~and~} P\in \left(P_{\inf}(X,\varphi),P_{\rm top}(X,\varphi)\right)$ such that
		  \[ P_{\mu}(X,\varphi)\neq P \text{~~for any~~} \mu\in \mathcal{M}_{\rm erg}(\Lambda).\]
	\end{itemize}
\end{theoremalph}

\begin{remark}
	We point out that in the dense part, the Dirac measure of the singularity is ``kicked out'' to be an isolated ergodic measure. The function $\varphi$ is chosen to be focused on the singularity, i.e. the value of $\varphi$ at the singularity is much larger than elsewhere.
\end{remark}

In $C^1$ topology, similar conclusion holds for a larger class of systems, {\it i.e.} singular hyperbolic attractors of higher dimensions.

\begin{theoremalph}\label{Thm:SH-attractor}
		There exist a dense $G_\delta$ subset $\mathcal{R}$ and a dense subset $\mathcal{D}$ in $\mathscr{X}^1(M)$ such that 
	\begin{itemize}
		\item {\bf generic systems:} if $\Lambda$ is a singular hyperbolic attractor of $X\in \mathcal{R}$, then $\Lambda$ admits the  intermediate pressure property;
		
		\item {\bf dense systems:} if $\Lambda$ is a singular hyperbolic attractor of $X\in \mathcal{D}$ with $\dim(E^{cu})=2$ where $T_\Lambda M=E^{ss}\oplus E^{cu}$ is the singular hyperbolic splitting of $\Lambda$, then $\Lambda$ does not admit the  intermediate pressure property.
	\end{itemize}
\end{theoremalph}

\begin{remark}
	Note that every singular hyperbolic attractor (in particular, every geometric Lorenz attractor) satisfies the star property and thus admits the intermediate entropy property by~\cite{LSWW}. In fact, every singular hyperbolic attractor is entropy expansive by~\cite{PYY} and thus admits measures of maximal entropy. As a consequence, for every singular hyperbolic attractor $\Lambda$ of a vector field $X\in\mathscr{X}^1(M)$ and every $h\in [0,h_{\rm top}(X,\Lambda)]$, there exists $\mu\in\mathcal{M}_{\rm erg}(\Lambda)$ such that $h_\mu(X)=h$. This gives a contrast between intermediate entropy property and intermediate pressure property for singular hyperbolic attractors.
\end{remark}

We point out that a strategy to verify intermediate pressure property for systems with specification-like properties was introduced by  Sun~\cite{Sun-2019} and was developed by Ji-Chen-Lin~\cite{JCL} who proved the intermediate pressure property for symbolic systems with non-uniform structure. However, the presence of  singularities constitutes an obstruction of specification-like property for geometric Lorenz attractors, see for instance~\cite{SVY15,Wen-Wen20}. To overcome this, we apply the horseshoe approximation property proved in~\cite{STVW} in the generic case.

\section{Preliminary}\label{Section:Pre}
In this section, we present some notions and known results.
\subsection{Topological pressure}\label{Section:pressure}
By a dynamical system $(K,f)$, we mean that $K$ is a compact metric space with metric $d(\cdot,\cdot)$ and $f\colon K\rightarrow K$ is a continuous transformation. Recall that we denote by $C(K,\mathbb{R})$ the space of all continuous functions from $K$ to $\mathbb{R}$ endowed with the  supremum norm $\|\cdot\|_\infty$.  
We follow~\cite[Section 9.1]{Wa} to define the topological pressure.

Given $n\in\mathbb{N}$ and $\varepsilon>0$, a subset $E$ of $K$ is called an $(n,\varepsilon)$-{\it separated set} if for any two distinct points $x, y\in E$, there exists $j\in [0,n-1]$ such that $d(f^j(x),f^j(y))>\varepsilon$.

Given a function $\varphi\in C(K,\mathbb{R})$, a point $x\in K$ and $n\in \mathbb{N}$, denote by 
\[S_n\varphi(x)=\sum_{j=0}^{n-1}\varphi(f^j(x)),\]
\[P_n(f,\varphi,\varepsilon)=\sup\left\{\sum_{x\in E}e^{S_n\varphi(x)}\colon \text{~$E$ is an $(n,\varepsilon)$-separated set}\right\}.\]
The {\it topological pressure} of $(K,f)$ with respect to $\varphi$ is defined as
\[P_{\rm top}(f,\varphi)=\lim_{\varepsilon\rightarrow 0}\limsup_{n\rightarrow\infty}\frac{1}{n}\log P_n(f,\varphi,\varepsilon).\]
In particular, the notion of pressure contains topological entropy in the sense that when considering   the $0$-function on $K$, one has $P_{\rm top}(f,0)=h_{\rm top}(f)$. 
Moreover, as we have mentioned before, by the variational principle~\cite[Theorem 9.10]{Wa} for pressure, one has that  
\[P_{\rm top}(f,\varphi)=\sup_{\mu\in\mathcal{M}_{\rm inv}(f)} P_\mu(f,\varphi)=\sup_{\mu\in\mathcal{M}_{\rm erg}(f)} P_\mu(f,\varphi),\]
where  
\[P_\mu(f,\varphi)=h_\mu(f)+\int \varphi {\rm d}\mu\]
is the  measure theoretical pressure of an invariant measure $\mu\in\mathcal{M}_{\rm inv}(f)$.
One can refer to~\cite[Chapter 9]{Wa} for more characterizations on topological pressure.

For a topological flow $\{X_t\}_{t\in\mathbb{R}}$ over $K$, its topological entropy and topological pressure are defined as the ones of its time-one map $X_1\colon K\rightarrow K$.

\subsection{Geometric Lorenz attractor and singular hyperbolicity} 

Let $M^d$ (or $M$ for simplicity) be a $d$-dimensional closed smooth Riemannian manifold.
Denote by $\mathscr{X}^r(M)$ the space of $C^r$ vector fields on $M$ where $r\in\mathbb{N}\cup\{\infty\}$. Given a vector field $X\in \mathscr{X}^r(M) $, denote by $\{X_t\}_{t\in\mathbb{R}}$ the $C^r$ flow generated by $X$ on $M$.
By an invariant set  of a vector field $X$, we mean that the set is invariant under the flow $\{X_t\}_{t\in\mathbb{R}}$.
For a point $x\in M$, denote by $\orb(x)$ its orbit under the flow $\{X_t\}_{t\in\mathbb{R}}$.
We denote by $\sing(X)$ the set of all singularities of $X$, i.e. $\sing(X)=\{\sigma\colon X(\sigma)=0\}$.

In this section, we devoted to give the definition of geometric Lorenz attractor following~\cite{guck,gw,williams}.
We first recall the notion of hyperbolicity and singular hyperbolicity.

\begin{definition}\label{Def:Hyp}
	Let $X\in \mathscr{X}^1(M)$. An invariant compact set $\Lambda$   of $X$ is \emph{hyperbolic} if there exist a continuous $DX_t$-invariant decomposition $T_\Lambda M=E^s\oplus \langle X\rangle \oplus E^u$ where $\langle X\rangle$ is the  linear space generated by $X$ and two constants  $C>1,\lambda>0$  such that for any $x\in\Lambda$ and any $t\geq 0$,
	\begin{itemize}
		\item $\|DX_t v\|< Ce^{-\lambda t}\|v\|$ for any $v\in E^s_x\setminus\{0\}$;
		
		\item $\|DX_{-t} v\|< Ce^{-\lambda t}\|v\|$ for any $v\in E^u_x\setminus\{0\}$.
	\end{itemize}
\end{definition}
Note that in Definition~\ref{Def:Hyp}, if $\Lambda$ is  a singularity $\sigma\in\sing(X)$, then $\langle X\rangle$ is reduced to $0$. In particular, if $\Lambda$ is a transitive set containing both singularities and regular points, then $\Lambda$ is not hyperbolic. To describe the structure of 3-dimensional robust transitive singular sets, Morales-Pacifico-Pujals~\cite{mpp} introduced the notion of singular hyperbolicity, which was extended to higher dimensions in~\cite{lgw,Me-Mo,zgw}. One can refer to~\cite{AP} for an overview of singular hyperbolic systems.

\begin{definition}\label{Def:SH}
	Let $X\in \mathscr{X}^1(M)$. An invariant compact set $\Lambda$   of $X$ is \emph{singular hyperbolic} if there exist a continuous $DX_t$-invariant decomposition $T_\Lambda M=E^{ss} \oplus E^{cu}$ and two constants  $C>1,\lambda>0$  such that for any $x\in\Lambda$ and any $t\geq 0$,
	\begin{itemize}
		\item $E^{ss}$ is dominated by $E^{cu}$: $\|DX_t|_{E^{ss}_x}\|\cdot \|DX_{-t}|_{E^{cu}_{X_t(x)}}\| <C e^{-\lambda t}$;
		
		\item $\|DX_t v\|< C e^{-\lambda t}\|v\|$ for any $v\in E^{ss}_x\setminus\{0\}$;
		
		\item $E^{cu}$ is sectionally expanding: $\left|{\rm det} (DX_t|_{V})\right|>C e^{\lambda t}$ for any two dimensional linear space $V\subset E^{cu}_x$.
	\end{itemize}
\end{definition} 

For a hyperbolic periodic point $p$ of a vector field $X\in\mathscr{X}^1(M)$, the classical theory of invariant manifolds by~\cite{hps} gives two submanifolds: the stable and unstable manifolds of $\orb(p)$ defined as
\[W^s(\orb(p))=\left\{x\colon \lim_{t\rightarrow+\infty}d(X_t(x),\orb(p))=0\right\};\]
\[W^u(\orb(p))=\left\{x\colon \lim_{t\rightarrow+\infty}d(X_{-t}(x),\orb(p))=0\right\}.\]
Two hyperbolic periodic points $p,q$ are called \emph{homoclinically related} if $W^s(\orb(p))$ has non-empty transverse intersections with $W^u(\orb(q))$  and vice versa. The \emph{homoclinic class} of a hyperbolic periodic point $p$, denote by $H(p)$, is the closure of the set of periodic points that are homoclinically related to $p$, equivalently, $H(p)$ is the closure of the set of transverse intersections between $W^s(\orb(q))$  and $W^u(\orb(q))$, see for instance~\cite{Wen-book}.

Now we are ready to give the definition of a geometric Lorenz attractor  following~\cite{guck,gw,williams}. First recall that an invariant compact set $\Lambda$ of a vector field $X\in\mathscr{X}^1(M)$ is an attractor if $\Lambda$ is a transitive set and there exists a neighborhood $U$ of $\Lambda$ such that $X_t(\overline U)\subset U$ for any $t>0$ and $\Lambda=\bigcap_{t\geq 0}X_t(\overline U)$. The neighborhood $U$ is called an {\it attracting neighborhood} of $\Lambda$. For a singularity $\sigma$ of a vector field $X$ and a constant $\delta>0$, the local stable manifold of $\sigma$ with size $\delta$ is defined as
\[W^s_{\delta}(\sigma)=\left\{x\colon \lim_{t\rightarrow+\infty}d(X_t(x),\sigma)=0 \text{~and~} d(X_t(x),\sigma)<\delta \text{~for $\forall t\geq 0$}\right\}.\]
We also denote by $W^s_{\rm loc}(\sigma)$ when one does not want to emphasize the size $\delta$.
The local unstable manifold $W^s_{\delta}(\sigma)$ or $W^s_{\rm loc}(\sigma)$ of $\sigma$ is defined similarly.

\begin{definition}\label{Def:Lorenz}
	Let $X\in\mathscr{X}^r(M^3), r\geq 1$. A singular hyperbolic attractor  $\Lambda$ of $X$ is a \emph{geometric Lorenz attractor} if the following properties are satisfied.
	\begin{enumerate}
		\item $\Lambda$ contains a unique singularity $\sigma$ with exponents $\lambda_1<\lambda_2<0<\lambda_3$ satisfying $\lambda_1+\lambda_3<0$ and $\lambda_2+\lambda_3>0$;
		
		\item $\Lambda$ admits a $C^r$-cross section $\Sigma$ which is $C^1$-diffeomorphic to $[-1,1]\times [-1,1]$ such that  for any point $z\in U\setminus W^s_{\rm loc}(\sigma)$, there exists $t>0$ such that $X_t(z)\in \Sigma$, where $U$ is an attracting neighborhood of $\Lambda$;
		
		\item Identifying $\Sigma=[-1,1]\times [-1,1]$, then the center line $\ell=\{0\}\times [-1,1]$ coincides with $W^s_{\rm loc}(\sigma)\cap \Sigma$ and the Poincar\'e return map $P\colon \Sigma\setminus\ell\rightarrow\Sigma$ is a $C^1$-smooth skew product 
		\[P(x,y)=(f(x,y),H(x,y)),\qquad \forall (x,y)\in \Sigma\setminus\ell\]
		with the following properties:
		\begin{itemize}
			\item $H(x,y)<0$ for $x>0$, $H(x,y)>0$ for $x<0$ and 
			\[\sup_{(x,y)\in\Sigma\setminus \ell} \max\left\{\left|\partial H(x,y)/\partial x\right|,\left|\partial H(x,y)/\partial y\right|\right\}<1\]
			
			\item the one-dimensional quotient map $f\colon [-1,1]\setminus\{0\}\rightarrow [-1,1]$ is $C^1$-smooth  satisfying
			$\lim_{x\rightarrow 0^-}f(x)=1,~~\lim_{x\rightarrow 0^+}f(x)=-1, -1<f(x)<1\text{~and~}  f'(x)>\sqrt{2}~\forall x\in [-1,1]\setminus \{0\}.$
		\end{itemize}
	\end{enumerate}
\end{definition}

\begin{figure}\label{Lorenz}
	\centering
	\includegraphics[width=10cm]{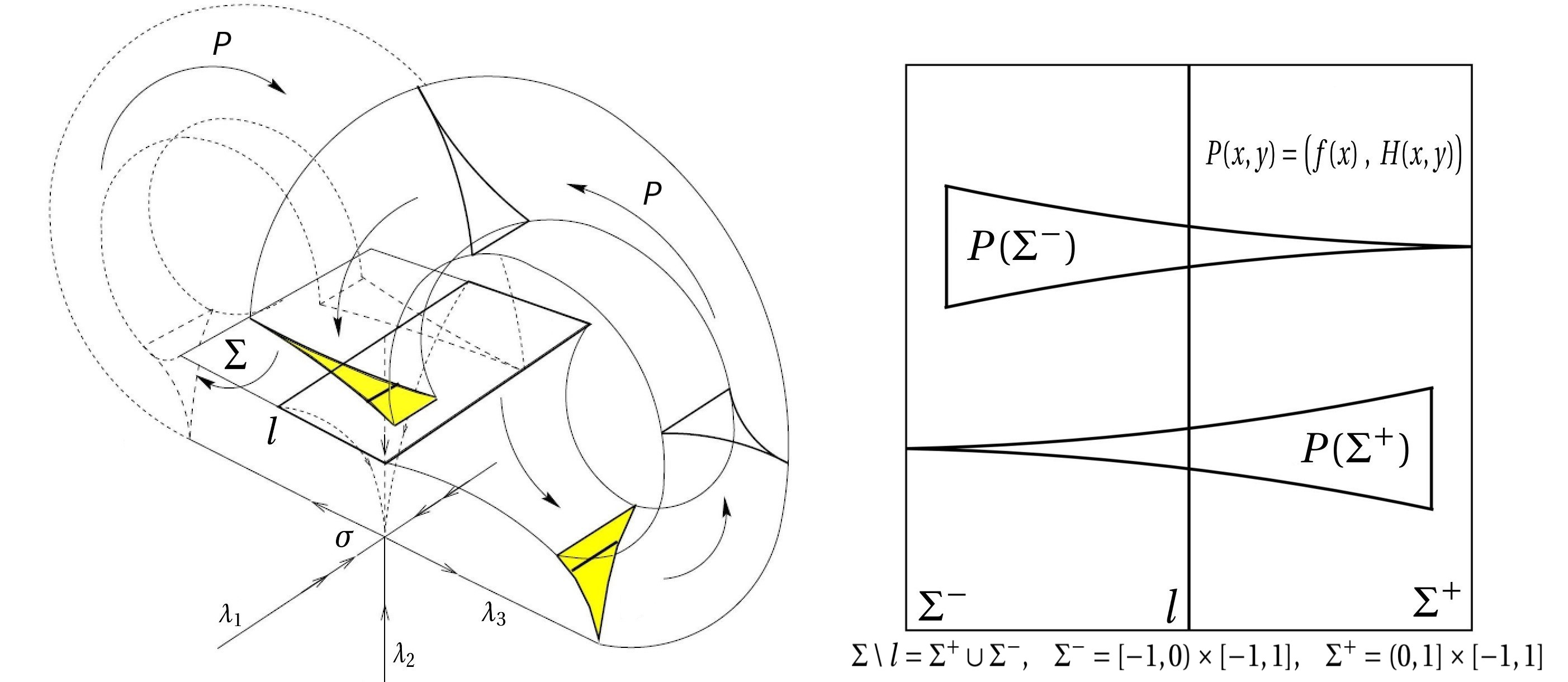}
	\caption{Geometric Lorenz attractor and the return map}
\end{figure}

A geometric Lorenz attractor is always a homoclinic class in which every two periodic points are homoclinically related~\cite[Theorem 6.8]{AP}, see also~\cite{Bau}. For every $r\in\mathbb{N}_{\geq 2}\cup\{\infty\}$, the set of vector fields $X\in\mathscr{X}^r(M^3)$ admitting a geometric Lorenz attractor forms an open set in $\mathscr{X}^r(M^3)$~\cite[Proposition 3.8]{STW}. We formulate theses properties as the following proposition.

\begin{proposition}\label{Prop:Lorenz}
	Assume $X\in \mathscr{X}^r(M^3), r\in\mathbb{N}_{\geq 2}\cup\{\infty\}$ admits a geometric Lorenz attractor $\Lambda$ with attracting neighborhood $U$. Then $\Lambda$ is a homoclinic class in which each pair of periodic points are homoclinically related. Moreover, there exists a neighborhood $\mathcal{V}$ of $X$ in $\mathscr{X}^r(M^3)$ such that for every $Y\in\mathcal{V}$, the maximal invariant compact set $\Lambda_Y=\bigcap_{t\geq 0} Y_t(U)$ is a geometric Lorenz attractor of $Y$ with $U$ being an attracting neighborhood.
\end{proposition}

%
%
%
%

\section{Ergodic theory for singular hyperbolic homoclinic classes}

Let $X\in\mathscr{X}^1(M)$ and $\Lambda$ be a non-trivial homoclinic class of $X$.
Denote by $\mathcal{M}(\Lambda)$ the space of probability measures supported on $\Lambda$  endowed with the weak$^*$ topology and let $d^*(\cdot,\cdot)$ be a metric of $\mathcal{M}(\Lambda)$ that is coherent with the weak$^*$ topology.
Denote by $\mathcal{M}_{\rm inv}(\Lambda)$ and $\mathcal{M}_{\rm erg}(\Lambda)$ respectively the space of invariant and ergodic probability measures of $X$ supported on $\Lambda$. 
Denote by 
\[\mathcal{M}_1(\Lambda)=\left\{\mu\in\mathcal{M}_{\rm inv}(\Lambda)\colon \mu({\rm Sing}(\Lambda))=0\right\}.\]
In this section, we always assume $\Lambda$ satisfies the following property:
\medskip

\noindent{\it Property (H): all periodic orbits contained in $\Lambda$ are homoclinically related.}

\medskip
The generic part of Theorem~\ref{Thm:Lorenz}~\&~\ref{Thm:SH-attractor} follows from the following theorem  which will be proved in this section. We apply the strategy in~\cite[Section 6]{Sun-2019} where the author proves that entropy genericity together with upper semi-continuity of entropy map implies pressure genericity, see also~\cite[Section 3]{JCL} where the authors apply  similar strategy to prove intermediate pressure property for symbolic systems with non-uniform structure.

\begin{theorem}\label{Thm:main-intermediate-pressure}
	Let $X\in\mathscr{X}^1(M)$ and $\Lambda$ be a non-trivial singular hyperbolic homoclinic class of $X$ satisfying Property (H). 
	Assume moreover that $\overline{\mathcal{M}_1(\Lambda)}=\mathcal{M}_{\rm inv}(\Lambda)$. 
	Given $\varphi\in C(\Lambda,\mathbb{R})$,  for any $P\in (P_{\inf}(X,\varphi),P_{\rm top}(X,\varphi))$, the set 
	\[\mathcal{P}_{\rm erg}(X,\varphi,P)\colon=\left\{\mu\in\mathcal{M}_{\rm erg}(\Lambda)\colon P_{\mu}(X,\varphi)=P\right\}\]
	forms a residual subset of the compact metric space
	\[\mathcal{P}(X,\varphi,P,+)\colon=\left\{\mu\in\mathcal{M}_{\rm inv}(\Lambda)\colon \int \varphi {\rm d}\mu\leq P\leq P_{\mu}(X,\varphi)\right\}. \]
	As a consequence, $\Lambda$ admits the intermediate pressure property.
\end{theorem}



Now we manage to prove the generic part of Theorem~\ref{Thm:Lorenz} \&~\ref{Thm:SH-attractor} based on Theorem~\ref{Thm:main-intermediate-pressure} whose proof will be given in Section~\ref{Section:proof-Theorem-3.1}. For completeness, we first state the generic part of~\cite[Theorem B~$\&$~B.1]{STW}.

\begin{theorem}[{Generic part of~\cite[Theorem B~$\&$~B.1]{STW}}]\label{Thm:generic-STW}
	There exists a dense $G_\delta$ subset $\mathcal{R}^r$ in $\mathscr{X}^r(M^3)$ for every  $r\in\mathbb{N}_{\geq 2}\cup\{\infty\}$  and a dense $G_\delta$ subset $\mathcal{R}\in\mathscr{X}^1(M)$ such that if $\Lambda$ is a geometric Lorenz attractor of  $X\in\mathcal{D}^r$ or $\Lambda$ is a singular hyperbolic attractor of  $X\in\mathcal{D}$, then 
    \[\overline{\mathcal{M}_{\rm per}(\Lambda)}=
		\overline{\mathcal{M}_{\rm erg}(\Lambda)}=		\mathcal{M}_{\rm inv}(\Lambda),\] 
	where  $\mathcal{M}_{\rm per}(\Lambda)$ denotes the set of all periodic measures supported on $\Lambda$.
		Moreover, the space $\mathcal{M}_{\rm erg}(\Lambda)$ is path connected.		
\end{theorem}

\begin{proof}[Proof of  Theorem~\ref{Thm:Lorenz} \&~\ref{Thm:SH-attractor}: generic part]
	Firstly, every geometric Lorenz attractor is a homoclinic class and satisfies Property $(H)$  by Proposition~\ref{Prop:Lorenz}. 
	By the main results in~\cite{CY-Robust-attractor}, for $C^1$ open and dense vector fields $X\in\mathscr{X}^1(M)$, every non-trivial singular hyperbolic attractor is a homoclinic class satisfying Property $(H)$.
	Thus combined by Theorem~\ref{Thm:generic-STW}, there exist a dense $G_\delta$ subset $\mathcal{R}^r$ in $\mathscr{X}^r(M^3)$ for every  $r\in\mathbb{N}_{\geq 2}\cup\{\infty\}$  and a dense $G_\delta$ subset $\mathcal{R}\in\mathscr{X}^1(M)$ such that if $\Lambda$ is a geometric Lorenz attractor of  $X\in\mathcal{D}^r$ or $\Lambda$ is a singular hyperbolic attractor of  $X\in\mathcal{D}$, then $\Lambda$ is a homoclinic class satisfying Property $(H)$ and moreover 
	\[\overline{\mathcal{M}_{\rm per}(\Lambda)}=\mathcal{M}_{\rm inv}(\Lambda).\]
	As a consequence, one has
	\[\overline{\mathcal{M}_{1}(\Lambda)}=\mathcal{M}_{\rm inv}(\Lambda).\]
	Thus the generic part of Theorem~\ref{Thm:Lorenz}~$\&$~\ref{Thm:SH-attractor} follows directly from Theorem~\ref{Thm:main-intermediate-pressure}.
%
\end{proof}

\subsection{Preliminary lemmas for entropy}

This subsection devotes to prove the following proposition.
\begin{proposition}\label{Pro:generic-entropy}
	Let  $X\in\mathscr{X}^1(M)$ and $\Lambda$ be a non-trivial singular hyperbolic homoclinic class of $X$ satisfying Property (H). 
	Assume moreover that $\overline{\mathcal{M}_1(\Lambda)}=\mathcal{M}_{\rm inv}(\Lambda)$. 
	Then for any $h\in[0, h_{\rm top}(\Lambda))$, the set
	\[\mathcal{M}_{\rm erg}(\Lambda,h)\colon=\left\{\mu\in \mathcal{M}_{\rm erg}(\Lambda)\colon h_\mu(X)=h\right\}\]
	forms a residual subset in the compact metric space
	\[\mathcal{M}(\Lambda,h,+)\colon=\left\{\mu\in \mathcal{M}_{\rm inv}(\Lambda)\colon h_\mu(X)\geq h\right\}.\]
\end{proposition}


We first recall the definition of horseshoe for vector fields.
\begin{definition}\label{Def:horseshoe}
	Given $X\in\mathscr{X}^1(M)$ and an invariant compact set $\Lambda$. We call $\Lambda$ a \emph{basic set} if it is transitive, hyperbolic and locally maximal. We call $\Lambda$ a \emph{horseshoe} if it is a non-trivial basic set, i.e. $\Lambda$ is not reduced to the single orbit of a  hyperbolic critical element, and the intersection between $\Lambda$ and any local cross section of the flow is totally disconnected.  
\end{definition}

\begin{lemma}\label{Lem:dense-intermediate-entropy}
	Let $X\in\mathscr{X}^1(M)$ and $\Lambda$ be a non-trivial singular hyperbolic homoclinic class of $X$ satisfying Property (H). Then for each $\mu\in\mathcal{M}_{\rm erg}(\Lambda)$ with $h_\mu(X)>0$, for each $h\in(0,h_\mu(X))$ and each $\varepsilon>0$, there exists $\nu\in\mathcal{M}_{\rm erg}(\Lambda)$ such that
	\[d^*(\nu,\mu)<\varepsilon \qquad\text{~and~}\qquad h_\nu(X)=h.\]
\end{lemma}
\begin{proof}
	Since $h_\mu(X)>0$, one has that $\mu(\sing(\Lambda))=0$.
	For any $h\in(0,h_\mu(X))$ and any $\varepsilon>0$, by~\cite[Lemma 4.10]{STVW}, there exist a horseshoe $\Lambda_{\varepsilon,h}\subset\Lambda$ such that 
	\begin{itemize}
		\item every $\omega\in\mathcal{M}_{\rm inv}(\Lambda)$ satisfies
		\[d^*(\omega,\mu)<\varepsilon;\]
		
		\item $h_{\rm top}(\Lambda_{h,\varepsilon})>h.$
	\end{itemize}
    Since every horseshoe satisfies the intermediate entropy property~\cite[Proposition 1.2]{LSWW}, for the constant $h\in (0,\Lambda_{h,\varepsilon})$, there exists $\nu\in\mathcal{M}_{\rm erg}(\Lambda_{h,\varepsilon})\subset \mathcal{M}_{\rm erg}(\Lambda)$ such that $h_\nu(X)=h$.
    The ergodic measure $\nu$ satisfies the conclusion.
\end{proof}

The following lemma states  the ``entropy dense'' property of $\mathcal{M}_{\rm inv}(\Lambda)$ provided $\Lambda$ being a singular hyperbolic homoclinic class with Property $(H)$.
\begin{lemma}\label{Lem:entropy-dense}{\rm\cite[Proposition 4.13]{STVW}}
	Let $X\in\mathscr{X}^1(M)$ and $\Lambda$ be a non-trivial singular hyperbolic homoclinic class of $X$ satisfying Property (H). Assume moreover that $\overline{\mathcal{M}_1(\Lambda)}=\mathcal{M}_{\rm inv}(\Lambda)$. Then 
	for any $\mu\in \mathcal{M}_{\rm inv}(\Lambda)$ and any $\varepsilon>0$, there exists  $\nu\in \mathcal{M}_{\rm erg}(\Lambda)$ such that 
	\[d^*(\nu, \mu)<\varepsilon \qquad\text{~and~}\qquad h_{\nu}(X)>h_\mu(X)-\varepsilon.\]
\end{lemma}
\begin{remark*}
	In fact, \cite[Proposition 4.13]{STVW} proves more that the measure $\mu$ is chosen to be supported on a horseshoe contained in $\Lambda$ which is called the ``horseshoe approximation property''.
\end{remark*}

Applying Lemma~\ref{Lem:dense-intermediate-entropy} and~\ref{Lem:entropy-dense}, one proves the following lemma. Recall for $h\geq 0$, we have defined in Proposition~\ref{Pro:generic-entropy}
\[\mathcal{M}(\Lambda,h,+)\colon=\left\{\mu\in \mathcal{M}_{\rm inv}(\Lambda)\colon h_\mu(X)\geq h\right\}.\]

\begin{lemma}\label{Lem:generic-entropy}
	Let $X\in\mathscr{X}^1(M)$ and $\Lambda$ be a non-trivial singular hyperbolic homoclinic class of $X$ satisfying Property (H). Assume moreover that $\overline{\mathcal{M}_1(\Lambda)}=\mathcal{M}_{\rm inv}(\Lambda)$.
	Then for any $0\leq h<h'$, the set 
	\[\mathcal{M}_{\rm erg}(\Lambda,h,h')=\mathcal{M}_{\rm inv}(\Lambda,h,h')\cap \mathcal{M}_{\rm erg}(\Lambda)\]
	is dense in $\mathcal{M}(\Lambda,h,+)$, where
	\[\mathcal{M}_{\rm inv}(\Lambda,h,h')=\{\mu\in \mathcal{M}_{\rm inv}(\Lambda)\colon h\leq h_\mu(X)<h'\}.\]
	
\end{lemma}

\begin{proof}
	Given $\mu\in \mathcal{M}(\Lambda,h,+)$, i.e. $h_\mu(X)\geq h$. Since $h<h_{\rm top}(\Lambda)$, one can take $\mu_0\in\mathcal{M}_{\rm erg}(\Lambda)$ such that $h_{\mu_0}(X)>h$. Given $\varepsilon>0$,  take $\theta\in (0,1)$, such that the invariant measure $\mu_1=\theta\cdot \mu_0+(1-\theta)\cdot\mu$ satisfies $d^*(\mu_1,\mu)<\varepsilon$.
	One has that 
	\[h_{\mu_1}(X)=\theta\cdot h_{\mu_0}(X)+(1-\theta)\cdot h_\mu(X)>h.\]
	By Lemma~\ref{Lem:entropy-dense}, there exists $\mu_2\in\mathcal{M}_{\rm erg}(\Lambda)$ such that 
	\[d^*(\mu_2,\mu_1)<\varepsilon \qquad\text{~and~}\qquad h_{\mu_2}(X)>h_{\mu_1}(X)-(h_{\mu_1}(X)-h)=h.\]
	Applying Lemma~\ref{Lem:dense-intermediate-entropy} to $\mu_2$ and $h''\in (h,\min\{h',h_{\mu_2}(X)\})$, there exists $\mu_3\in \mathcal{M}_{\rm erg}(\Lambda)$ such that 
	\[d^*(\mu_3,\mu_2)<\varepsilon \qquad\text{~and~}\qquad h_{\mu_3}(X)=h''\in (h,h').\]
	Thus $d^*(\mu_3,\mu)<3\cdot\varepsilon$ and $\mu_3\in \mathcal{M}_{\rm erg}(\Lambda,h,h')$.
	The denseness of $\mathcal{M}_{\rm erg}(\Lambda,h,h')$ in $\mathcal{M}(\Lambda,h,+)$ is obtained since $\varepsilon$ can be chosen arbitrarily.
\end{proof}

Now we are ready to prove Proposition~\ref{Pro:generic-entropy}.
\begin{proof}[Proof of Proposition~\ref{Pro:generic-entropy}]
	Given $h\in[0,h_{\rm top}(\Lambda))$, the set $\mathcal{M}(\Lambda,h,+)$ is non-empty by the variational principle for entropy.
	By~\cite[Theorem A]{PYY}, the entropy map 
	\[h\colon \mathcal{M}_{\rm inv}(\Lambda)\rightarrow \mathbb{R},~~~~h\colon \mu\mapsto h_{\mu}(X)\]
	is upper semi-continuous because of the singular hyperbolicity.
	As a consequence, the set $\mathcal{M}(\Lambda,h,+)$ forms a compact metric subspace of $\mathcal{M}_{\rm inv}(\Lambda)$, which is a Baire space.
	On the other hand, the set $\mathcal{M}_{\rm erg}(\Lambda)$ forms a $G_\delta$ subset of $\mathcal{M}_{\rm inv}(\Lambda)$ by~\cite[Proposition 5.1]{abc}, and thus is also a Baire space.
	\begin{claim}\label{Claim:residual}
		For any $h'>h$, the set $\mathcal{M}_{\rm erg}(\Lambda,h,h')$ forms a residual subset in $\mathcal{M}(\Lambda,h,+)$.
	\end{claim}
	\begin{proof}[Proof of Claim~\ref{Claim:residual}]
		Firstly, by the upper semi-continuity of the entropy map, the set
		\[\mathcal{M}_{\rm inv}(\Lambda,h,h')=\left\{\mu\in \mathcal{M}_{\rm inv}(\Lambda)\colon h\leq h_\mu(X)<h'\right\}= \mathcal{M}(\Lambda,h,+)\setminus \mathcal{M}(\Lambda,h',+)\]
		is relatively open in $\mathcal{M}(\Lambda,h,+)$. 
		Thus 
		\[\mathcal{M}_{\rm erg}(\Lambda,h,h')=\mathcal{M}_{\rm inv}(\Lambda,h,h')\cap \mathcal{M}(\Lambda,h,+)\cap \mathcal{M}_{\rm erg}(\Lambda)\]
		 forms a $G_\delta$ subset of $\mathcal{M}(\Lambda,h,+)$.		
		On the other hand, by Lemma~\ref{Lem:generic-entropy}, the set $\mathcal{M}_{\rm erg}(\Lambda,h,h')$ is dense in $\mathcal{M}(\Lambda,h,+)$.
		Thus $\mathcal{M}_{\rm erg}(\Lambda,h,h')$ is  residual  in $\mathcal{M}(\Lambda,h,+)$.
	\end{proof}
    Note that
    \[\mathcal{M}_{\rm erg}(\Lambda,h)=\bigcap_{k\geq 1} \mathcal{M}_{\rm erg}(\Lambda,h,h+1/k).\]
    By Claim~\ref{Claim:residual}, one has that $\mathcal{M}_{\rm erg}(\Lambda,h)$ forms a residual subset in $\mathcal{M}(\Lambda,h,+)$ which proves Proposition~\ref{Pro:generic-entropy}.    
\end{proof}

\subsection{Ergodic measures with intermediate pressure}

This section devotes to proving the following proposition, which is key to the proof of Proposition~\ref{Thm:main-intermediate-pressure} in  Section~\ref{Section:proof-Theorem-3.1}.
\begin{proposition}\label{Prop:generic-pressure}
	Let $X\in\mathscr{X}^1(M)$ and $\Lambda$ be a non-trivial singular hyperbolic homoclinic class of $X$ satisfying Property (H). Assume moreover that $\overline{\mathcal{M}_1(\Lambda)}=\mathcal{M}_{\rm inv}(\Lambda)$.
	Given  $\varphi\in C(\Lambda, \mathbb{R})$ and $P\in (P_{\inf}(X,\varphi),P_{\rm top}(X,\varphi))$, 	 for any $P'>P$, the set
	\[\mathcal{P}_{\rm erg}(X,\varphi,P,P')\colon=\left\{\mu\in\mathcal{M}_{\rm erg}(\Lambda)\colon \int \varphi {\rm d}\mu \leq P\leq P_{\mu}(X,\varphi)<P'\right\}\]
	is dense in the set
	\[\mathcal{P}(X,\varphi,P,+)\colon=\left\{\mu\in\mathcal{M}_{\rm inv}(\Lambda)\colon \int \varphi {\rm d}\mu \leq P\leq P_{\mu}(X,\varphi)\right\}.\]
\end{proposition}

\begin{proof}
	Given $\mu\in \mathcal{P}(X,\varphi,P,+)$, i.e. $\int \varphi {\rm d}\mu \leq P\leq P_{\mu}(X,\varphi)\}$ and given $\eta>0$, we will show that there exists $\nu\in \mathcal{P}_{\rm erg}(X,\varphi,P,P')$ such that $d^*(\nu,\mu)<\eta$.
	We consider the following cases.
	
	\paragraph{Case I: $P< P_{\mu}(X,\varphi)$.}
	We consider two sub-cases whether $\int \varphi {\rm d}\mu<P$ or $\int \varphi {\rm d}\mu=P$.
	
	\paragraph{Subcase I.1: $\int \varphi {\rm d}\mu<P< P_{\mu}(X,\varphi)$.}
	We point out that this is the essential case to prove while the other cases would be reduced to this case.
	Let 
	\[a=\frac{1}{4}\min\left\{P-\int\varphi {\rm d}\mu,~~P_{\mu}(X,\varphi)-P,~~P'-P\right\}>0.\]
	Take $\eta_0\in(0,\eta)$ such that for any $\omega\in\mathcal{M}_{\rm inv}(\Lambda)$ satisfying $d^*(\omega,\mu)<\eta_0$, one has 
	\[-a<\int \varphi {\rm d}\omega-\int \varphi {\rm d}\mu<a.\]
	Let
	\[h=\min\left\{P_{\mu}(X,\varphi),~~P' \right\}-\int\varphi {\rm d}\mu-a.\]
	Then one has
	\[h=\left(\min\left\{P_{\mu}(X,\varphi),~~P' \right\}-P\right)+\left(P-\int\varphi {\rm d}\mu\right)-a\geq 8a-a>0 \text{~~~and}\]
	\[h<P_{\mu}(X,\varphi)-\int\varphi {\rm d}\mu=h_\mu(X),\]
	which implies $\mu\in \mathcal{M}(X,h,+)$, where recall that 
	\[\mathcal{M}(\Lambda,h,+)=\left\{\mu\in \mathcal{M}_{\rm inv}(\Lambda)\colon h_\mu(X)\geq h\right\}.\]
	By Proposition~\ref{Pro:generic-entropy}, there exists $\nu\in\mathcal{M}_{\rm erg}(\Lambda,h)$, i.e. $\nu\in\mathcal{M}_{\rm erg}(\Lambda)$ and $h_\nu(X)=h$ such that $d^*(\nu,\mu)<\eta_0<\eta$, thus we have
	\[\int \varphi {\rm d}\mu-a<\int \varphi {\rm d}\nu<\int \varphi {\rm d}\mu+a<P.\]
	On the other hand, one has
	\begin{align*}
		P_\nu(X,\varphi)&=h_{\nu}(X)+\int \varphi {\rm d}\nu\\
		&=h+\int \varphi {\rm d}\nu\\
		&=\min\left\{P_{\mu}(X,\varphi),~~P' \right\}-\left(\int\varphi {\rm d}\mu+a-\int \varphi {\rm d}\nu\right)\\
		&<P';
	\end{align*}
    and
    \begin{align*}
    	P_\nu(X,\varphi)&=\min\left\{P_{\mu}(X,\varphi),~~P' \right\}-\int\varphi {\rm d}\mu-a+\int \varphi {\rm d}\nu\\
    	&=\min\left\{P_{\mu}(X,\varphi),~~P' \right\}-\left(\int\varphi {\rm d}\mu-\int \varphi {\rm d}\nu\right)-a\\
    	&>\min\left\{P_{\mu}(X,\varphi),~~P' \right\}-2a\\
    	&>P.
    \end{align*}
	Thus one has 
	\[\int \varphi {\rm d}\nu<P<P_\nu(X,\varphi)<P'\]
	which implies $\nu\in\mathcal{P}_{\rm erg}(X,\varphi,P,P')$.
	Moreover, recall that $\nu$ is chosen to satisfy $d^*(\nu,\mu)<\eta$, thus $\nu$ is what we need in this subcase.
	
	\paragraph{Subcase I.2: $\int \varphi {\rm d}\mu=P< P_{\mu}(X,\varphi)$.}
	Since $P>P_{\inf}(X,\varphi)$, one can take $\mu_1\in\mathcal{M}_{\rm inv}(\Lambda)$ such that $P_{\mu_1}(X,\varphi)<P$. Thus for $\theta\in(0,1)$, the measure $\mu_\theta=\theta \cdot\mu_1+(1-\theta)\cdot\mu$ satisfies
	\[\lim_{\theta\rightarrow 0^+} \mu_\theta=\mu, \text{~~and~~} P_{\mu_\theta}(X,\varphi)=\theta \cdot P_{\mu_1}(X,\varphi)+(1-\theta)\cdot P_{\mu}(X,\varphi).\]
	Thus one can take $\theta_1\in (0,1)$ small enough such that the measure $\mu_{\theta_1}$ satisfies
	\[d^*(\mu_{\theta_1},\mu)<\eta/2  \text{~~and ~~} P_{\mu_{\theta_1}}(X,\varphi)>P.\]
	On the other hand, since $\int \varphi {\rm d}\mu=P$ and $\int \varphi {\rm d}\mu_1\leq P_{\mu_1}(X,\varphi)<P$, thus one has
	\[\int \varphi {\rm d}\mu_{\theta_1}=\theta_1 \cdot \int \varphi {\rm d}\mu_1+(1-\theta_1)\cdot \int \varphi {\rm d}\mu<P.\]
	The above estimates means that 
	\[\int \varphi {\rm d}\mu_{\theta_1}<P<P_{\mu_{\theta_1}}(X,\varphi).\]
	This implies that the measure $\mu_{\theta_1}$ satisfies the assumptions in Subcase I.1.
	Thus for the measure $\mu_{\theta_1}$, there exists $\nu\in \mathcal{P}_{\rm erg}(X,\varphi,P,P')$ such that $d^*(\nu,\mu_{\theta_1})<\eta/2$.
	As a consequence, the measure $\nu\in \mathcal{P}_{\rm erg}(X,\varphi,P,P')$ satisfies
	\[d^*(\nu,\mu)\leq d^*(\nu,\mu_{\theta_1})+d^*(\mu_{\theta_1},\mu)<\eta.\]

	\paragraph{Case II: $P= P_{\mu}(X,\varphi)$.}
	Firstly, since $P\in(P_{\inf}(X,\varphi),P_{\rm top}(X,\varphi))$, there exist $\mu_2\in\mathcal{M}_{\rm inv}(\Lambda)$ such that $P_{\mu_2}(X,\varphi)>P=P_{\mu}(X,\varphi)$.
	We also consider the sub-cases whether $\int \varphi {\rm d}\mu<P$ or $\int \varphi {\rm d}\mu=P$.
	
	\paragraph{Subcase II.1: $\int \varphi {\rm d}\mu<P= P_{\mu}(X,\varphi)$.}
	Similarly as in Subcase I.2, when $\theta_2\in(0,1)$ is small enough, the measure
	$\mu_{\theta_2}=\theta_2\cdot\mu_2+(1-\theta_2)\cdot\mu$ satisfies
	\[d^*(\mu_{\theta_2},\mu)<\eta/2 \text{~~and~~} \int \varphi {\rm d}\mu_{\theta_2}<P.\]
	On the other hand,
	\[P_{\mu_{\theta_2}}(X,\varphi)=\theta_2 \cdot P_{\mu_2}(X,\varphi)+(1-\theta_2)\cdot P_{\mu}(X,\varphi)>P\]
	Thus $\mu_{\theta_2}$  satisfies the assumption as in Subcase 1.1 that
	\[\int \varphi {\rm d}\mu_{\theta_2}<P<P_{\mu_{\theta_2}}(X,\varphi).\]
	Then there exists $\nu\in \mathcal{P}_{\rm erg}(X,\varphi,P,P')$ such that $d^*(\nu,\mu_{\theta_2})<\eta/2$.
	As a consequence, 
    \[d^*(\nu,\mu)\leq d^*(\nu,\mu_{\theta_2})+d^*(\mu_{\theta_2},\mu)<\eta.\]

	\paragraph{Subcase II.2: $\int \varphi {\rm d}\mu=P= P_{\mu}(X,\varphi)$.}
	For the measure $\mu_2\in\mathcal{M}_{\rm inv}(\Lambda)$ with $P_{\mu_2}(X,\varphi)>P=P_{\mu}(X,\varphi)$,
	we consider two cases whether $\int \varphi {\rm d}\mu_2\leq P$ or $\int \varphi {\rm d}\mu_2> P$.
	\medskip
	
	When  $\displaystyle\int \varphi {\rm d}\mu_2\leq P$, by taking $\theta_2\in(0,1)$  small enough, the measure
	$\mu_{\theta_2}=\theta_2\cdot\mu_2+(1-\theta_2)\cdot\mu$ satisfies
	\[d^*(\mu_{\theta_2},\mu)<\eta/2  \text{~~and~~} \int \varphi {\rm d}\mu_{\theta_2}\leq P<P_{\mu_{\theta_2}}(X,\varphi)\]
	Then $\mu_{\theta_2}$ satisfies the assumptions in Case I, thus there there exists $\nu\in \mathcal{P}_{\rm erg}(X,\varphi,P,P')$ such that $d^*(\nu,\mu_{\theta_2})<\eta/2$.
	As a consequence, 
	$d^*(\nu,\mu)\leq d^*(\nu,\mu_{\theta_2})+d^*(\mu_{\theta_2},\mu)<\eta.$
	\medskip
	
	The last case is when $\displaystyle\int \varphi {\rm d}\mu_2> P$. Take $\mu_1\in\mathcal{M}_{\rm inv}(\Lambda)$ such that $\displaystyle\int\varphi {\rm d}\mu_1\leq P_{\mu_1}(X,\varphi)<P$ as in Subcase I.2. 
	One can assume that $\displaystyle\int\varphi {\rm d}\mu_1< P_{\mu_1}(X,\varphi)$ which is equivalent to $h_{\mu_1}(X)>0$.
	Otherwise $\displaystyle\int\varphi {\rm d}\mu_1= P_{\mu_1}(X,\varphi)$ which is equivalent to $h_{\mu_1}(X)=0$. One takes $\mu'\in\mathcal{M}_{\rm inv}(\Lambda)$ such that $h_{\mu'}(X)>0$. (Such $\mu'$ must exist because $\Lambda$ is a non-trivial homoclinic class and thus $h_{\rm top}(\Lambda)>0$.)
	Then when $t\in(0,1)$ is small enough, the measure 
	\[\mu_t=t\cdot \mu'+(1-t)\mu_1 \text{~~would satisfy~~}
	\displaystyle\int\varphi {\rm d}\mu_t< P_{\mu_t}(X,\varphi)<P. \]
	As $\displaystyle\int\varphi {\rm d}\mu_1< P<\displaystyle\int\varphi {\rm d}\mu_2$, there exists $\mu_3$ which is a non-trivial convex sum of $\mu_1$ and $\mu_2$ such that $\displaystyle\int\varphi {\rm d}\mu_3= P$ and $h_{\mu_3}(X)>0$ (because $h_{\mu_1}(X)>0$ and $h_{\mu_2}(X)\geq 0$).
	Thus $\mu_3$ satisfies 
	\[\displaystyle\int\varphi {\rm d}\mu_3= P<P_{\mu_3}(X,\varphi).\]
	As before, consider the convex sum $\mu_\theta=\theta\cdot \mu_3+(1-\theta)\cdot\mu$ for $\theta\in(0,1)$. When $\theta_3\in(0,1)$ is small enough, the measure $\mu_{\theta_3}$ satisfies
	\[d^*(\mu_{\theta_3},\mu)<\eta/2 \qquad\text{~and~}\qquad \int\varphi {\rm d}\mu_{\theta_3}= P<P_{\mu_{\theta_3}}(X,\varphi).\]
	By Subcase I.2, there exists $\nu\in \mathcal{P}_{\rm erg}(X,\varphi,P,P')$ such that $d^*(\nu,\mu_{\theta_3})<\eta/2$.
	As a consequence, one has
	\[d^*(\nu,\mu)\leq d^*(\nu,\mu_{\theta_3})+d^*(\mu_{\theta_3},\mu)<\eta.\]
		The proof is completed.
\end{proof}

\subsection{Proof of Theorem~\ref{Thm:main-intermediate-pressure}}\label{Section:proof-Theorem-3.1}
Applying Proposition~\ref{Prop:generic-pressure}, one manages to prove Theorem~\ref{Thm:main-intermediate-pressure} and thus the generic part of  Theorem~\ref{Thm:Lorenz} \&~\ref{Thm:SH-attractor} would be proved. 
\begin{proof}[Proof of Theorem~\ref{Thm:main-intermediate-pressure}]
	Let $X,\Lambda,\varphi$ be as in Theorem~\ref{Thm:main-intermediate-pressure}. Take $P\in(P_{\inf}(X,\varphi),P_{\rm top}(X,\varphi))$.
	Then there exists $\mu_1,\mu_2\in\mathcal{M}_{\rm inv}(\Lambda)$ such that 
	\[P_{\mu_1}(X,\varphi)<P<P_{\mu_2}(X,\varphi).\]
	Since the pressure map $\mu\mapsto P_{\mu}(X,\varphi)$ is affine, there exists $\theta\in(0,1)$ such that the measure
	$\mu=\theta\cdot\mu_1+(1-\theta)\cdot\mu_2$ satisfies 
	\[P_{\mu}(X,\varphi)=\int \varphi {\rm d}\mu+h_{\mu}(X)=P.\]
	Thus 
	\[\int \varphi {\rm d}\mu=P-h_{\mu}(X)\leq P=P_{\mu}(X,\varphi),\]
	which shows that $\mu\in \mathcal{P}(X,P,+)$ and as a consequence $\mathcal{P}(X,P,+)\neq\emptyset$.
	On the other hand
	\[\mathcal{P}(X,\varphi,P,+)=\left\{\mu\in\mathcal{M}_{\rm inv}(\Lambda)\colon \int \varphi {\rm d}\mu\leq P\right\}\bigcap \left\{\mu\in\mathcal{M}_{\rm inv}(\Lambda)\colon P_{\mu}(X,\varphi)\geq P\right\} \]
	is closed because the map $\mu\mapsto \int \varphi {\rm d}\mu$ is continuous and the pressure map is upper semi-continuous by~\cite[Theorem A]{PYY}. Thus $\mathcal{P}(X,\varphi,P,+)$ is a non-empty compact subset of the Baire space $\mathcal{M}_{\rm inv}(\Lambda)$.
	Theorem~\ref{Thm:main-intermediate-pressure} is essentially proved by the following claim.
	\begin{claim}\label{Claim:generic-pressure}
		For any $P'>P$, the set $\mathcal{P}_{\rm erg}(X,\varphi,P,P')$ forms a residual subset in $\mathcal{P}(X,\varphi,P,+)$.
	\end{claim}
    \begin{proof}[Proof of Claim~\ref{Claim:generic-pressure}]
    	Note that $\mathcal{P}(X,\varphi,P,+)\setminus \mathcal{P}(X,\varphi,P',+)$ forms a relatively open  set in $\mathcal{P}(X,\varphi,P,+)$ and  $\mathcal{M}_{\rm erg}(\Lambda)$ is a $G_\delta$ subset of $\mathcal{M}_{\rm inv}(\Lambda)$. Thus the set
    	\[\mathcal{P}_{\rm erg}(X,\varphi,P,P')=\mathcal{M}_{\rm erg}(\Lambda)\cap \left(\mathcal{P}(X,\varphi,P,+)\setminus \mathcal{P}(X,\varphi,P',+)\right)\]
    	is a $G_\delta$ subset of $\mathcal{P}(X,\varphi,P,+)$.
    	
    	On the other hand, Proposition~\ref{Prop:generic-pressure} shows $\mathcal{P}_{\rm erg}(X,\varphi,P,P')$ is dense in $\mathcal{P}(X,\varphi,P,+)$.
    	Thus $\mathcal{P}_{\rm erg}(X,\varphi,P,P')$ forms a residual subset in $\mathcal{P}(X,\varphi,P,+)$.
    \end{proof}
    Note that 
   \[ \mathcal{P}_{\rm erg}(X,\varphi,P)=\bigcap_{k\geq 1} \mathcal{P}_{\rm erg}(X,\varphi,P,P+1/k).\]
   Claim~\ref{Claim:generic-pressure} implies that $\mathcal{P}_{\rm erg}(X,\varphi,P)$ forms a residual subset in $\mathcal{P}(X,\varphi,P,+)$.
   This completes the proof of Theorem~\ref{Thm:main-intermediate-pressure}. 
\end{proof}

\section{Dense part of Theorem~\ref{Thm:Lorenz} \&~\ref{Thm:SH-attractor}}

For the dense part of Theorem~\ref{Thm:Lorenz} \&~\ref{Thm:SH-attractor}, one first applies~\cite{STW} that for $C^r(r\geq 2)$-dense Lorenz attractors, the singular measure is isolated in the ergodic measure space. Then one takes a continuous function $\varphi$ such that its value at the singularity is ``much larger'' than elsewhere. This leads that the intermediate pressure property fails for such a $\varphi$.
For completeness, we state the dense part of~\cite[Theorem B~$\&$~B.1]{STW}. 

\begin{theorem}[{Dense part of~\cite[Theorem B~$\&$~B.1]{STW}}]\label{Thm:dense-STW}
	There exists a dense  subset $\mathcal{D}^r$ in $\mathscr{X}^r(M^3)$ for every  $r\in\mathbb{N}_{\geq 2}\cup\{\infty\}$  and a dense subset $\mathcal{D}\in\mathscr{X}^1(M)$ such that if $\Lambda$ is a geometric Lorenz attractor of  $X\in\mathcal{D}^r$ or $\Lambda$ is a singular hyperbolic attractor of  $X\in\mathcal{D}$ with $\dim(E^{cu})=2$, then 
	\[\overline{\mathcal{M}_{\rm per}(\Lambda)}\subsetneqq \overline{\mathcal{M}_{\rm erg}(\Lambda)}\subsetneqq \mathcal{M}_{\rm inv}(\Lambda). \]
	Moreover, the space $\mathcal{M}_{\rm erg}(\Lambda)$ is not connected and
	the singular measure $\delta_\sigma$ is isolated inside $\mathcal{M}_{\rm erg}(\Lambda)$ where $\sigma\in\sing(\Lambda)$.
\end{theorem}

We point out that the singular measure $\delta_\sigma$  is isolated inside $\mathcal{M}_{\rm erg}(\Lambda)$ was not originally stated in~\cite[Theorem B~$\&$~B.1]{STW}. 
In fact the strategy of  the dense part of~\cite[Theorem B~$\&$~B.1]{STW} is to kick out the singular measure $\delta_\sigma$  to be isolated in $\mathcal{M}_{\rm erg}(\Lambda)$. 
One can see~\cite[Remark 1.3~$\&$~Section 4.4]{STW} for details.

\begin{proof}[Dense part of Theorem~\ref{Thm:Lorenz} \&~\ref{Thm:SH-attractor}]
    Let $\mathcal{D}^r\subset\mathscr{X}^r(M^3)$ for $r\in\mathbb{N}_{\geq 2}\cup\{\infty\}$ and $\mathcal{D}\in\mathscr{X}^1(M)$ be taken from Theorem~\ref{Thm:dense-STW}. Assume $\Lambda$ is a geometric Lorenz attractor of  $X\in\mathcal{D}^r$ or $\Lambda$ is a singular hyperbolic attractor of  $X\in\mathcal{D}$ with $\dim(E^{cu})=2$, then the singular measure $\delta_\sigma$ is isolated inside $\mathcal{M}_{\rm erg}(\Lambda)$, where $\sigma\in\sing(\Lambda)$. Let $\eta>0$ be a small constant such that the open ball $B(\sigma,2\eta)$ satisfies $\mu(B(\sigma,2\eta))<1/4$ for all $\mu\in\mathcal{M}_{\rm erg}(\Lambda)\setminus \{\delta_\sigma\}$.
	
	Take a constant $L>4\cdot h_{\rm top}(\Lambda)$ and take a continuous function $\varphi\colon\Lambda\rightarrow \mathbb{R}$ satisfying the following properties:
	\begin{itemize}
		\item $\varphi(x)=L$ for every $x\in\Lambda\cap B(\sigma,\eta)$;
		\item $\varphi(x)=0$ for every $x\in\Lambda\setminus B(\sigma,2\eta)$;
		\item $0\leq \varphi(x)\leq L$ for every $x\in\Lambda$.
	\end{itemize}
    For the potential $\varphi$, note first that 
    \[P_{\delta_\sigma}(X, \varphi)=h_{\delta_\sigma}(X)+\int \varphi {\rm d} \delta_\sigma=0+\varphi(\sigma)=L.\]
    On the other hand, for any $\mu\in\mathcal{M}_{\rm erg}(\Lambda)\setminus \{\delta_\sigma\}$,
    \begin{align*}
    	P_{\mu}(X, \varphi)
    	&=h_{\mu}( X)+\int \varphi {\rm d} \mu\\
    	&\leq  h_{\rm top}(\Lambda)+\int_{\Lambda\cap U_2} \varphi {\rm d} \mu+\int_{\Lambda\setminus U_2} \varphi {\rm d} \mu\\
    	&\leq L/4+ L/4+0=L/2\\
    	&<P_{\delta_\sigma}(X, \varphi).
    \end{align*}
    Thus $\Lambda$ does not admit the intermediate pressure property. 
\end{proof}

\begin{remark*}
	We give two remarks here for the dense part of of Theorem~\ref{Thm:Lorenz} \&~\ref{Thm:SH-attractor}.
	\begin{enumerate}
		\item It is obvious that the singular measure $\delta_\sigma$ is the unique equilibrium of the function $\varphi$ constructed in the proof of dense part. Moreover, if a sequence of invariant measures $\{\mu_n\}_{n=1}^{\infty}$ satisfies $h_{\mu_n}(X)+\int \varphi {\rm d}\mu_n\rightarrow P_{\rm top}(X,\varphi)$, then one has $\mu_n\rightarrow\delta_\sigma$ in the weak*-topology. By~\cite[Corollary 4 \& Corollary 8]{Walters-91}, $P_{\rm top}(X,\cdot)$ is Fr\'echet differentiable\footnote{Please refer to \cite{Walters-91} for the definition of Fr\'echet differentiablity.} at $\varphi$.
		
		\item Following the same argument in the dense part, if a dynamical system $(X,f)$ has an isolated ergodic measure $\mu$ such that $\supp(\mu)\neq X$, then $(X,f)$ does not admit the intermediate pressure property.
	\end{enumerate}	
\end{remark*}

\bibliographystyle{plain}

\flushleft{\bf Yi Shi} \\
School of Mathematics, Sichuan University, Chengdu, 610065, China\\
\textit{E-mail:} \texttt{shiyi@scu.edu.cn}\\

\flushleft{\bf Xiaodong Wang} \\
School of Mathematical Sciences,  CMA-Shanghai, Shanghai Jiao Tong University, Shanghai, 200240,  China\\
\textit{E-mail:} \texttt{xdwang1987@sjtu.edu.cn}\\

\end{document}